\documentclass{amsart}
\usepackage{times,mathrsfs,array,amssymb,pb-diagram}

\newcounter{lemma}

\newtheorem{Lemma}[lemma]{Lemma}

\newtheorem{Proposition}[lemma]{Proposition}

\theoremstyle{definition}

\def\C{\mathbb C}
\def\H{\mathbb H}
\def\O{\mathcal O}

\def\Q{\mathbb Q}
\def\R{\mathbb R}
\def\Z{\mathbb Z}
\def\new{\mathrm{new}}

\def\FL#1{\left\lfloor #1\right\rfloor}

\def\SL{\mathrm{SL}}

\def\M#1#2#3#4{\begin{pmatrix}#1&#2\\#3&#4\end{pmatrix}}
\def\SM#1#2#3#4{\left(\begin{smallmatrix}#1&#2\\#3&#4\end{smallmatrix}
  \right)}

\begin{document}

\title{Computing modular equations for Shimura curves}

\author{Yifan Yang}
\address{Department of Applied Mathematics, National Chiao Tung
  University and National Center for Theoretical Sciences, Hsinchu,
  Taiwan 300}  
\email{yfyang@math.nctu.edu.tw}
\date{\today}

\begin{abstract} In the classical setting, the modular equation of
  level $N$ for the modular curve $X_0(1)$ is the polynomial relation
  satisfied by $j(\tau)$ and $j(N\tau)$, where $j(\tau)$ is the
  standard elliptic $j$-function. In this paper, we will describe a
  method to compute modular equations in the setting of Shimura
  curves. The main ingredient is the explicit method for computing
  Hecke operators on the spaces of modular forms on Shimura curves
  developed in \cite{Yang-Schwarzian}.
\end{abstract}

\thanks{The author would like to thank Professor John McKay and
  Professor John Voight for providing valuable comments on the
  paper. The author was partially supported by Grant
  99-2115-M-009-011-MY3 of the National Science Council, Taiwan
  (R.O.C.).}
\maketitle

\begin{section}{Introduction}
Let $j(\tau)$ be the elliptic $j$-function. For a positive integer
$N$, consider the set
\begin{equation*}
\begin{split}
  \Gamma_N&=\SL(2,\Z)\M100N\SL(2,\Z) \\
  &=\left\{\M abcd:~a,b,c,d\in\Z,~ad-bc=N,~\gcd(a,b,c,d)=1\right\}.
\end{split}
\end{equation*}
The group $\SL(2,\Z)$ acts on $\Gamma_N$ by multiplication on the
left and it can be easily checked that the cardinality of
$\SL(2,\Z)\backslash\Gamma_N$ is finite. Moreover, the multiplication
of any element in $\SL(2,\Z)$ on the right of the cosets in
$\SL(2,\Z)\backslash\Gamma_N$ simply permutes the cosets. Thus, any
symmetric sum of $j(\gamma\tau)$, where $\gamma$ runs over a complete
set of representatives for $\SL(2,\Z)\backslash\Gamma_N$, will be a
modular function on $\SL(2,\Z)$ and therefore can be written as a
rational function of $j(\tau)$. In fact, because the only possible
pole occurs at the cusp, this rational function of $j(\tau)$ is
actually a polynomial. Then the \emph{modular polynomial}
$\Phi_N(x,y)$ of level $N$ is defined to be the polynomial in
$\C[x,y]$ such that
$$
  \Phi_N(x,j(\tau))=\prod_{\gamma\in\SL(2,\Z)\backslash\Gamma_N}
  \left(x-j(\gamma\tau)\right).
$$
In fact, it can be easily proved that this polynomial $\Phi_N(x,y)$
is in $\Z[x,y]$. Since $Y_0(1)$ is the moduli space of isomorphism
classes of elliptic curves over $\C$, the equation $\Phi_N(x,y)=0$ is
the \emph{modular equation} of level $N$ for moduli of elliptic curves
over $\C$. That is, if $E_1$ and $E_2$ are two elliptic curves over $\C$
admitting a cyclic $N$-isogeny between them, then their $j$-invariants
$j(E_1)$ and $j(E_2)$ satisfy $\Phi_N(j(E_1),j(E_2))=0$.

Observe that if the Fourier expansion of $j(\tau)$ is
$q^{-1}+744+196884q+\cdots$, $q=e^{2\pi i\tau}$, then the Fourier
expansion of $j(N\tau)$ is simply $q^{-N}+744+196884q^N+\cdots$. Thus,
in principle, to determine $\Phi_N(x,y)$, one just has to compute
enough Fourier coefficients for $j(\tau)$ and solve a system of linear
equations. Of course, the main difficulty in practice is that the
coefficients are gigantic. On the other hand, because the coefficients
of $\Phi_N(x,y)$ are all integers, one can compute the reduction of
$\Phi_N(x,y)$ modulo $p$ for a suitable number of primes $p$ and then
use the Chinese remainder theorem to recover the coefficients.
See \cite{Lauter} for the current state of the art in the computation
of $\Phi_N(x,y)$.
% One cannot go very far, even with the
%computing power of the present day.

In this paper, we shall consider modular equations in the settings of
Shimura curves. Let $B$ be an indefinite quaternion algebra of
discriminant $D>1$ over $\Q$ and $\O$ be a maximal order in $B$.
Choose an embedding $\iota:B\hookrightarrow M(2,\R)$ and let
$$
  \Gamma(\O):=\{\iota(\alpha):~\alpha\in\O,
 ~\mathit{N}(\alpha)=1\}
$$
be the image of the norm-one group of $\O$ under $\iota$, where
$\mathit{N}(\alpha)$ denotes the reduced norm of $\alpha\in B$. Then
the Shimura curve $X_0^D(1)$ is defined to be the quotient space
$\Gamma(\O)\backslash\H$. As shown in \cite{Shimura-CM1}, the Shimura
curve $X_0^D(1)$ is the coarse moduli space for isomorphism classes of
abelian surfaces with quaternionic multiplication (QM) by $\O$. Let
$W_D$ denote the group of Atkin-Lehner involutions on $X_0^D(1)$. For
our purpose, we will also consider the quotient curves of $X_0^D(1)$
by subgroups $W$ of $W_D$.

Now assume that $X_0^D(1)$ has genus $0$ and choose a Hauptmodul
$t(\tau)$ for $X_0^D(1)$ so that $t(\tau)$ generates the function
field on $X_0^D(1)$. For a positive integer $N$ relatively prime to
$D$, pick an element $\alpha$ of norm $N$ in $\O$ such that
$\O\cap(\alpha^{-1}\O\alpha)$ is an Eichler order of level $N$.
Then the \emph{modular polynomial} of level $N$ for $X_0^D(1)$ is
defined to be the polynomial $\Phi_N^D(x,y)$ of minimal degree, up to
scalars, such that $\Phi_N^D(t(\tau),t(\iota(\alpha)\tau))=0$, which
is essentially the rational function $\widetilde\Phi^D_N(x,y)$ such
that
$$
  \widetilde\Phi_N^D(x,t(\tau))=\prod_{\gamma\in\Gamma(\O)\backslash
  \Gamma(\O)\iota(\alpha)\Gamma(\O)}(x-t(\gamma\tau)).
$$
Here unlike the case of the modular curve $X_0(1)$, a symmetric sum
of $t(\gamma\tau)$ as $\gamma$ runs through representatives of
$\Gamma(\O)\backslash\Gamma(\O)\iota(\alpha)\Gamma(\O)$ is not
equal to a polynomial of $t(\tau)$ in general.
Since $X_0^D(1)$ is the moduli space of abelian surfaces over $\C$
with quaternionic multiplication, the modular equation of level $N$
relates the moduli of two abelian surfaces with QM that have a certain
type of isogenies between them. (The precise description of the
isogeny is a little complicated to be given here. See \cite[Appendix
A]{Molina1} for details.) Modular equations for Atkin-Lehner quotients
$X_0^D(1)/W$ are similarly defined.

When $D>1$, the problem of explicitly determining modular equations
for Shimura curves is significantly more complicated than its
classical counterpart. The reasons are that Shimura curves do not have
cusps and it is difficult to find the Taylor expansion of an
automorphic function with respect to a local parameter at any given
point on the Shimura curve. When $D$ and $N$ are both small such that
$X_0^D(N)$ is also of genus $0$, it is possible to work out an
explicit cover $X_0^D(N)\to X_0^D(1)$ from the ramification data
alone. Then from the explicit cover, one can compute the modular
equation of level $N$. This has been done in \cite{Elkies} for a
limited number of cases.

Another possible method to compute modular equations for Shimura
curves uses the Schwarzian differential equation associated to a
Hauptmodul $t(\tau)$. (See Section \ref{section: Schwarzian} for a
review on the notion and properties of Schwarzian differential
equation.) The idea is that with a properly chosen pair of solutions
$F_1(t)$ and $F_2(t)$ of the Schwarzian differential equation and a
correct positive integer $e$, the expression $(F_2(t)/F_1(t))^e$ can
be taken to be a local parameter at the point $\tau_0$ of the Shimura
curve with $t(\tau_0)=0$. Inverting the expression, one gets the
Taylor expansion of $t$ with respect to the local parameter. If
somehow one manages to find the Taylor expansion of
$t(\iota(\alpha)\tau)$ (this can always be done when
$t(\iota(\alpha)\tau_0)=0$, then by computing enough terms and
solving a system of linear equations, one gets the modular equation.
However, as far as we know, there does not seem to be any paper in the
literature that employs this idea to obtain modular equations for
Shimura curves. (The paper \cite{Bayer-Travesa} did obtain the Taylor
expansion for the Hauptmodul in the case $D=6$. More recently, Voight
and Willis \cite{Voight-Willis} developed a method for numerically
computing Taylor expansions of automorphic forms on Shimura curves.)

In his Ph.D. thesis \cite{Voight-thesis}, Voight developed a method to
compute modular equations for Shimura curves associated to quaternion
algebras over totally real number fields in the cases when the Shimura
curves have genus zero and precisely three elliptic points, i.e.,
Shimura curves associated to arithmetic triangle groups. The idea
is to use hypergeometric functions, which are essentially solutions of
the Schwarzian differential equations mentioned above, to numerically
determine coordinates of CM-points and then use the existence of
canonical models and explicit Shimura reciprocity laws to determine
modular equations. Even though the equations are obtained numerically,
they can be verified rigorously by showing that the monodromy group of
the branched cover is correct.

In this paper, we shall present a new method to compute modular
equations for Shimura curves. Our method also uses Schwarzian
differential equations, but relies more heavily on the arithmetic side
of the theory. Namely, in an earlier work \cite{Yang-Schwarzian}, we
showed that the spaces of automorphic forms of any given weight on a
Shimura curve $X_0^D(1)$ of genus $0$ can be completely characterized
in terms of solutions of the Schwarzian differential equation. We
\cite{Yang-Schwarzian} then devised a method to compute Hecke
operators with respect to our basis. The Jacquet-Langlands
correspondence plays a crucial role in our approach. It turns out that
our method and results in \cite{Yang-Schwarzian} can also be used to
compute modular equations for Shimura curves. We will describe the
procedure in Section \ref{section: method} and give a detailed example
in Section \ref{section: example}. Using the computer algebra
system MAGMA \cite{Magma}, we have succeeded in determining modular
equations of prime level up to $19$ for $X_0^6(1)/W_6$ and those of
prime level up to $23$ for $X_0^{10}(1)/W_{10}$.

% As a standard
%application of modular equations, we also determine the coordinates of
%the Shimura curves $X_0^6(1)/W_6$ and $X_0^{10}(1)/W_{10}$ at certain
%CM-points, verifying most entries in Tables 1 and 3 of \cite{Elkies}.

The main difficulty in generalizing our method to general Shimura
curves lies at the fact that our method requires that a Schwarzian
differential equation is known beforehand. Because of the problem of
the existence of accessory parameters, the determination of Schwarzian
differential equations usually requires that an explicit cover
$X_0^D(N)\to X_0^D(1)$ is known, which can be problematic when $D$ is
large. (In some sense, what we do here is to deduce modular
equations of higher levels from that of a given small level.)
Nonetheless, once an explicit cover of Shimura curves is 
determined, the combination of the methods in \cite{Yang-Schwarzian}
and in this paper will yield modular equations for the Shimura curve.

The rest of the paper is organized as follows. In Section
\ref{section: Schwarzian}, we review the definition and properties of
Schwarzian differential equations. In Section \ref{section: method},
we describe our method to compute modular equations for Shimura curves
$X_0^D(1)/W$, assuming that an explicit cover $X_0^D(p_0)/W\to
X_0(1)/W$ is known. In Section \ref{section: example}, we give a
detailed example illustrating our method. In Sections \ref{section: 6}
and \ref{section: 10}, we list part of our computational results for
the Shimura curves $X_0^6(1)/W_6$ and $X_0^{10}(1)/W_{10}$. Files in
the MAGMA-readable format containing all our computational results are
available upon request.
\end{section}

\begin{section}{Schwarzian differential equations}
\label{section: Schwarzian}

In this section, we will review the definition and properties of
Schwarzian differential equations. In particular, assuming the Shimura
curve has genus $0$, we will recall the characterization of the spaces
of automorphic forms in terms of solutions of the associated
Schwarzian differential equation. The method for computing Hecke
operators on these spaces is too complicated to describe here. We
refer the reader to \cite{Yang-Schwarzian} for details. Most
materials in this section are taken from \cite{Yang-Schwarzian}.

Let $X$ be a Shimura curve. We assume that the associated quaternion
algebra is not $M(2,\Q)$ so that $X$ has no cusps. Let $F(\tau)$ be a
(meromorphic) automorphic form of weight $k$ and $t$ be a non-constant
automorphic function on $X$. It is known since the nineteen century
that the functions $F,\tau F,\ldots,\tau^kF$, as functions of $t$, are
solutions of a certain $(k+1)$-st linear ordinary differential
equations with algebraic functions as coefficients. (See
\cite{Stiller,Yang}.) In particular, since $t'(\tau)$ is a meromorphic
automorphic form of weight $2$, the function $t'(\tau)^{1/2}$ as a
function of $t$, satisfies a second-order linear ordinary differential
equation. We call this differential equation the \emph{Schwarzian
  differential equation} associated to $t$, which has the following
properties.

\begin{Proposition}[{\cite[Proposition 6]{Yang-Schwarzian}}]
\label{proposition: Q}
  Let $X$ be a Shimura curve of genus zero with elliptic
  points $\tau_1,\ldots,\tau_r$ of order $e_1,\ldots,e_r$,
  respectively. Let $t(\tau)$ be a Hauptmodul of $X$ and set
  $a_i=t(\tau_i)$, $i=1,\ldots,r$. Then $t'(\tau)^{1/2}$, as a
  function of $t$, satisfies the differential equation
  $$
    \frac{d^2}{dt^2}F+Q(t)F=0,
  $$
  where
  $$
    Q(t)=\frac14\sum_{j=1,a_j\neq\infty}^r\frac{1-1/e_j^2}{(t-a_j)^2}
        +\sum_{j=1,a_j\neq\infty}^r\frac{B_j}{t-a_j}
  $$
  for some constants $B_j$. Moreover, if $a_j\neq\infty$ for all $j$,
  then the constants $B_j$ satisfy
  $$
  \sum_{j=1}^r B_j=
  \sum_{j=1}^r\left(a_jB_j+\frac14(1-1/e_j^2)\right)=
  \sum_{j=1}^r\left(a_j^2B_j+\frac12a_j(1-1/e_j^2)\right)=0.
  $$
  Also, if $a_r=\infty$, then $B_j$ satisfy
  $$
  \sum_{j=1}^{r-1}B_j=0, \qquad
  \sum_{j=1}^{r-1}\left(a_jB_j+\frac14(1-1/e_j^2)\right)=\frac14(1-1/e_r^2).
  $$
\end{Proposition}

We remark that when the Shimura curve has genus $0$ and precisely $3$
elliptic points, the relations among $B_j$ are enough to determine the
constants $B_j$. This reflects the fact in classical analysis that a
second-order Fuchsian differential equation with exactly three
singularities is completely determined by the local exponents. When
the Shimura curve has more than $3$ elliptic points, the relations are
not enough to determine $B_j$. In literature, we refer to this kind of
situations by saying that \emph{accessory parameters} exist. In order
to determine the accessory parameters, one usually tries to find an
explicit cover of Shimura curves and use it to determine the
Schwarzian differential equations associated to the two curves
simultaneously.

Now one of the key observations in \cite{Yang-Schwarzian} is that the
analytic behavior of the Hauptmodul $t'(\tau)$ is very easy to
determine and from this, one can work out a basis for the space of
automorphic forms of even weight $k$ in terms of $t'(\tau)$.

\begin{Proposition}[{\cite[Theorem 4]{Yang-Schwarzian}}]
\label{proposition: basis}
  Assume that a Shimura curve $X$ has genus zero with elliptic
  points $\tau_1,\ldots,\tau_r$ of order $e_1,\ldots,e_r$,
  respectively. Let $t(\tau)$ be a Hauptmodul of $X$ and set
  $a_i=t(\tau_i)$, $i=1,\ldots,r$. For a positive even integer $k\ge 4$,
  let
  $$
    d_k=\dim S_k(X)=1-k+\sum_{j=1}^r\left\lfloor
    \frac k2\left(1-\frac1{e_j}\right)\right\rfloor.
  $$
  Then a basis for the space of automorphic forms of weight $k$ on $X$ is
  $$
    t'(\tau)^{k/2}t(\tau)^j\prod_{i=1,a_i\neq\infty}^r
    \left(t(\tau)-a_i\right)^{-\lfloor k(1-1/e_i)/2\rfloor}, \quad
    j=0,\ldots,d_k-1.
  $$
\end{Proposition}

The combination of two propositions shows that all automorphic forms
of a given even weight $k$ can be expressed in terms of the solutions
of the Schwarzian differential equation. In \cite{Yang-Schwarzian},
the author developed a method for computing Hecke operators relative
to the basis in Proposition \ref{proposition: basis}. The key
ingredients are the Jacquet-Langlands correspondence and explicit
covers of Shimura curves. We refer the reader to
\cite{Yang-Schwarzian} for details.
\end{section}

\begin{section}{Computing modular equations for Shimura curves}
\label{section: method}
In this section, we will present a method for computing modular
equations for Shimura curves, under the working assumptions that the
Schwarzian differential equations have been determined and, for a
certain prime fixed $p_0$, the matrices for the Hecke operator $T_{p_0}$
with respect to the bases in Proposition \ref{proposition: basis} have
already been computed for sufficiently many $k$ according to the
recipe in \cite{Yang-Schwarzian}.

Let $X$ be a Shimura curve of genus $0$ of the form $X_0^D(1)/W$ for
some subgroup $W$ of the group of Atkin-Lehner involutions and
$\Gamma$ be the discrete subgroup of $\SL(2,\R)$ corresponding to
$X$. For a prime $p$ not dividing $D$, let
$\gamma_j=\SM{a_j}{b_j}{c_j}{d_j}\in\mathrm{GL}^+(2,\R)$,
$j=0,\ldots,p$, be representatives for the cosets defining the Hecke
operator $T_p$, that is,
\begin{equation} \label{equation: Hecke}
  T_p:f\to p^{k/2-1}\sum_{j=0}^p
  \frac{(\det\gamma_j)^{k/2}}{(c_j\tau+d_j)^k}f(\gamma_j\tau).
\end{equation}
Then the modular equation $\Phi_p(x,y)=0$ of level $p$ is the
polynomial relation between a Hauptmodul $t(\tau)$ for $X$ and
$t(\gamma_j\tau)$ for arbitrary $j$.

Pick any nonzero automorphic form $F(\tau)$ on $X$ with the smallest
possible weight $k$. For convenience, we let
\begin{equation} \label{equation: Fj}
  F_j(\tau)=\frac{(\det\gamma_j)^{k/2}}{(c_j\tau+d_j)^k}F(\gamma_j\tau),
\end{equation}
the summands in \eqref{equation: Hecke}. Then any symmetric sum of
$F_j(\tau)/F(\tau)$, $j=0,\ldots,p$, will be an automorphic function
on $X$ and hence can be expressed as a rational function of
$t(\tau)$. In other words, there is a rational function
$\Psi(z,x)$ such that
\begin{equation} \label{equation: Psi}
  \Psi(z,t(\tau))=\prod_{j=0}^p
  \left(z-\frac{F_j(\tau)}{F(\tau)}\right).
\end{equation}
That is, we have
\begin{equation} \label{equation: sketch 1}
  \Psi\left(\frac{F_j(\tau)}{F(\tau)},t(\tau)\right)=0
\end{equation}
for all $j$.

\begin{Lemma} Let $\Psi(z,x)$ be the rational function defined by
  \eqref{equation: Psi}. Let $\Psi'(z,y)$ be the rational function
  such that
  $$
    \Psi'(z,y)=z^{p+1}\Psi(1/z,y),
  $$
  which is a polynomial in $z$. Let $R(x,y)$ be the resultant of
  $\Psi(z,x)$ and $\Psi'(z,y)$ with respect to the variable $z$.
  Then the modular polynomial $\Phi_p(x,y)$ appears as one of the
  irreducible factors over $\C$ in the numerator of the rational
  function $R(x,y)$.
\end{Lemma}

\begin{proof} Let $\gamma_j'=(\det\gamma_j)\gamma_j^{-1}$. There is
an integer $j'$, $0\le j'\le p$, such that
$\gamma_j'\in\Gamma\gamma_{j'}$. We apply the action of $\gamma_j'$ to
\eqref{equation: sketch 1} and get
$$
  \Psi\left(\frac{F(\tau)}{F_{j'}(\tau)},t(\gamma_{j'}\tau)\right)=0.
$$
In other words, we have
$$
  \Psi'\left(\frac{F_{j'}(\tau)}{F(\tau)},t(\gamma_{j'}\tau)\right)=0
$$
From this equality and \eqref{equation: sketch 1}, we see that if we
let $R(x,y)$ be the resultant of $\Psi(z,x)$ and $\Psi'(z,y)$ with
respect to the variable $z$, then $R(t(\tau),t(\gamma_{j'}\tau))=0$.
In particular, the modular polynomial $\Phi_p(x,y)$ appears as an
irreducible factor of $R(x,y)$ over $\C$. This proves the lemma.
\end{proof}

We remark that, in general, there will be more than one irreducible
factors in the numerator of the resultant $R(x,y)$, but it is not
difficult to determine which one corresponds to $\Phi_p(x,y)$. For
example, we can use the fact that the roots of $\Phi_p(x,x)$ are the
coordinates of certain CM-points (also known as \emph{singular
  moduli}) to test which factor of $R(x,y)$ is $\Phi_p(x,y)$. Thus,
from the above discussion, we see that the most critical part of the
calculation is the determination of the rational function $\Psi(z,x)$,
which we address now.

From Newton's identity, we know that the problem of determining the
symmetric sums of $F_j(\tau)/F(\tau)$ in \eqref{equation: Psi} is
equivalent to that of determining $(F_j/F)^m$ for $m=1,\ldots,p+1$.
Observe that the sum of $F_j^m$ is equal to $p^{1-km/2}T_pF$, by the
definition of the Hecke operator $T_p$. Thus, to determine
$(F_j/F)^m$, one just needs to know how the Hecke operator $T_p$ acts
on the basis given in Proposition \ref{proposition: basis}. This
is where the work of \cite{Yang-Schwarzian} comes into play.

In \cite{Yang-Schwarzian}, we developed a method to compute $T_{p_0}$,
$p_0$ a prime not dividing $D$, for arbitrary weight $k$ with respect
to the basis given in Proposition \ref{proposition: basis}, assuming
an explicit cover $X_0^D(p_0)/W\to X_0^D(1)/W$ is known. Now recall
the following explicit version of the Jacquet-Langlands
correspondence.

\begin{Proposition}[{\cite{Jacquet-Langlands,Shimizu}}]
  \label{proposition: Jacquet-Langlands}
  Let $D$ be discriminant of an indefinite quaternion algebra over
  $\Q$. Let $N$ be a positive integer relatively prime to $D$. For an
  Eichler order $\O=\O(D,N)$ of level $(D,N)$ and a positive even
  integer, let $S_k(\Gamma(\O))$ denote the space of automorphic forms
  on the Shimura curve $X_0^D(N)$. Then
  $$
    S_k(\Gamma(\O))\simeq
    S_k^{D\text{-\rm{new}}}(DN):=\bigoplus_{d|N}\bigoplus_{m|N/d}
    S_k^\new(\Gamma_0(dD))^{[m]}
  $$
  as Hecke modules. Here
  $$
    S_k^\new(\Gamma_0(dD))^{[m]}=\{f(m\tau):~f(\tau)\in
    S_k^\new(\Gamma_0(dD))\}
  $$
  and $S_k^\new(\Gamma_0(dD))$ denotes the newform subspace of
  cusp forms of weight $k$ on $\Gamma_0(dD)$. In other words, for each
  Hecke eigenform $f(\tau)$ in $S_k^{D\text{-\rm{new}}}(\Gamma_0(DN))$, there
  corresponds a Hecke eigenform $\widetilde f(\tau)$ in
  $S_k(\Gamma(\O))$ that shares the same Hecke eigenvalues. Moreover,
  for a prime divisor $p$ of $D$, if the Atkin-Lehner involution $W_p$
  acts on $f$ by $W_pf=\epsilon_p f$, then
  $$
    W_p\widetilde f=-\epsilon_p\widetilde f.
  $$
\end{Proposition} 

For the situation under our consideration, we have
\begin{equation} \label{equation: Jacquet-Langlands}
  S_k^{\mathrm{new}}(\Gamma_0(D),W)\simeq S_k(\Gamma),
\end{equation}
where the left-hand side denotes the subspace of the
newform subspace of cusp forms of weight $k$ on $\Gamma_0(D)$
that has eigenvalue $-1$ for all $w_q\in W$, $q$ primes, and the
right-hand side denotes the space of automorphic forms of weight $k$
on $X$. Thus, assuming the Hecke operator $T_{p_0}$ on the 
space $S_k^{\mathrm{new}}(\Gamma_0(D),W)$ has no repeated
eigenvalues, one can obtain the matrices for $T_p$ with respect to the
bases in Proposition \ref{proposition: basis} from those for $T_{p_0}$
and the Fourier coefficients of Hecke eigenforms in
$S_k^{\mathrm{new}}(\Gamma_0(D),W)$.

In summary, to compute the modular equation $\Phi_p(x,y)$, we follow
the following steps, assuming that the Schwarzian differential equation
associated to $X=X_0^D(1)/W$ and an explicit cover $X_0^D(p_0)/W\to
X_0^D(1)/W$ are known for some prime $p_0$ not dividing $D$.
\begin{enumerate}
\item[(a)] Pick a nonzero automorphic form $F$ on $X$ of the smallest
  possible weight $k$, expressed in the form given in Proposition
  \ref{proposition: basis}.
\item[(b)] Compute the matrices for $T_{p_0}$ with respect to the
  basis in Proposition \ref{proposition: basis} for weights
  $k,2k,\ldots,(p+1)k$ using the method in \cite{Yang-Schwarzian}.
\item[(c)] Compute Fourier coefficients of Hecke eigenforms in the
  space $S_k^\new(\Gamma_0(D),W)$ in \eqref{equation:
    Jacquet-Langlands} to the precision of $p$ terms (using MAGMA
  \cite{Magma} or SAGE \cite{Sage}).
\item[(d)] Compute the matrices for $T_p$ with respect to the basis in
  Proposition \ref{proposition: basis} using informations from Steps
  (b) and (c). This gives us the expressions of $\sum_{j=0}^p(F_j/F)^m$
  in terms of $t(\tau)$, where $F_j$ are defined by \eqref{equation:
    Fj}.
\item[(e)] Use Newton's identity to convert expressions for
  $\sum_{j=0}^p(F_j/F)^m$ to those for symmetric sums of $F_j/F$ and
  hence determine the rational function $\Psi(z,x)$ in
  \eqref{equation: Psi}.
\item[(f)] Set $\Psi'(z,y)=z^{p+1}\Psi(1/z,y)$. Compute and factorize
  the resultant $R(x,y)$ of $\Psi(z,x)$ and $\Psi'(z,y)$ with respect
  to the variable $z$.
\item[(h)] Determine which irreducible factor of the numerator of
  $R(x,y)$ is the modular polynomial $\Psi_p(x,y)$ by using the fact
  that the roots of $\Psi_p(x,x)$ are coordinates of some CM-points
  (also known as singular moduli) on $X$.
\end{enumerate}

We now work out an example in details.
\end{section}

\begin{section}{An example}
\label{section: example}
In this section, we will work out the modular equation of level $7$
for the Shimura curve $X=X_0^{10}(1)/W_{10}$. (Note that this case was not
covered in \cite{Elkies}.)

The curve $X$ has $4$ elliptic points of orders $2,2,2,3$, which we
denote by $P_2,P_2',P_2'',P_3$, respectively. These are CM-points of
discriminants $-8$, $-20$, $-40$, and $-3$, respectively. According to
\cite{Elkies}, there is a Hauptmodul $t(\tau)$ that takes values
$\infty$, $2$, and $27$, and $0$, respectively. Using the covering
$X_0^{10}(3)/W_{10}\to X_0^{10}(1)/W_{10}$, we \cite[Equation
(9)]{Yang-Schwarzian} found that the Schwarzian differential equation
associated to $t(\tau)$ is
$$
  \frac{d^2}{dt^2}F+\frac{3t^4-119t^3+3157t^2-7296t+10368}
    {16t^2(t-2)^2(t-27)^2}F=0.
$$
Thus, by Proposition \ref{proposition: Q}, near the point $P_3$, the
$t$-expansion of $t'(\tau)$ is the square of a linear combination of
two solutions 
\begin{equation*}
\begin{split}
  F_1(t)&=t^{1/3}\left(1-\frac{10}{81}t-\frac{18539}{839808}t^2
 -\frac{168605}{25509168}t^3-\frac{107269219465}{46548313473024}t^4
 +\cdots\right), \\
  F_2(t)&=t^{2/3}\left(1-\frac5{81}t-\frac{99095}{5878656}t^2
 -\frac{8353325}{1428513408}t^3-\frac{851170821485}{385081502367744}t^4
 +\cdots\right)
\end{split}
\end{equation*}
of the differential equation above. Thus, by Proposition
\ref{proposition: basis}, a basis for the space of automorphic forms
of even weight $k$ on $X$ is
\begin{equation} \label{equation: 10 basis}
   f_{k,j}=\frac{t^{j-1}\left(F_1(t)-CF_2(t)\right)^{k/2}}
   {t^{\FL{k/3}}(1-t/2)^{\FL{k/4}}(1-t/27)^{\FL{k/4}}}, \quad
   j=1,\ldots, d_k=1-k+3\FL{\frac k4}+\FL{\frac k3},
\end{equation}
for some complex number $C$. In particular, the one-dimensional space
of automorphic forms of weight $4$ on $X$ is spanned by
$$
  F(\tau)=f_{4,1}(t(\tau)).
$$
Let $F_j(\tau)$, $j=0,\ldots,7$, be defined by \eqref{equation: Fj}
with $p=7$. According to the recipe described in Section \ref{section:
  method}, we first need to compute the matrices for $T_3$ with
respect to the basis in \eqref{equation: 10 basis} for weights $4m$,
$m=1,\ldots,8$. This has already been done in
\cite[Appendix D]{Yang-Schwarzian}. We found that the matrices of the
Hecke operators $T_3$ are
$$ \extrarowheight3pt
\begin{array}{c|l} \hline\hline
k & A_k \\ \hline
4 & -8 \\
8 & 28 \\
12& \begin{pmatrix}
  468  & -98 \\
 -1728 & 136 \end{pmatrix} \\
16& \begin{pmatrix}
  1728 &  490 \\
 34560 &-3572 \end{pmatrix} \\
%18& -14976 \\
20& \begin{pmatrix}
   -2268 &  -2450 \\
 -328320 &  35992 \end{pmatrix} \\
%22& -21924 \\
24& \begin{pmatrix}
  227772 & -272244 &   14406 \\
 -388800 & -258192 &   12250 \\
 2985984 &  711936 & -199556 \end{pmatrix}\\
%26& 162864 \\
28& \begin{pmatrix}
     420552 &   949620 &  -72030 \\
    -933120 &  4479732 &  -61250 \\
 -104509440 & 31147200 & -196568 \end{pmatrix} \\
%30& \begin{pmatrix}
% -6676344 & 4593750 \\
% 14541120 & 3031596 \end{pmatrix} \\
32& \begin{pmatrix}
   29821932 &   -5456052 &   360150 \\
   95084928 &  -48253536 &   306250 \\
 1803534336 & -618444288 & 19290988 \end{pmatrix} \\ \hline\hline
\end{array}
$$
That is, for the integer $k$ in the table, we have
$$
  T_3\begin{pmatrix}f_{k,1}\\ \vdots \\ f_{k,d_k}\end{pmatrix}
 =A_k\begin{pmatrix}f_{k,1}\\ \vdots \\ f_{k,d_k}\end{pmatrix}.
$$

The next informations we need are the Fourier expansions of Hecke
eigenforms in the space $S_k^\new(\Gamma_0(10),-1,-1)$ for
$k=4,8,\ldots,32$. By MAGMA \cite{Magma}, they are
\begin{equation*}
\begin{split}
  &q + 2q^2 - 8q^3 + \cdots - 4q^7 + \cdots, \\
  &q + 2^3q^2 + 28q^3 + \cdots + 104q^7 + \cdots, \\
  &q + 2^5q^2 + aq^3 + \cdots + (-177a + 60500)q^7 + \cdots, \\
  &q + 2^7q^2 + aq^3 + \cdots + (423a - 102460)q^7 + \cdots, \\
  &q + 2^9q^2 + aq^3 + \cdots + (-417a + 48562100)q^7 + \cdots, \\
  &q + 2^{11}q^2 + aq^3 + \cdots +
   \frac1{24}(-a^2 + 156628a + 145233723936)q^7 + \cdots, \\
  &q + 2^{13}q^2 + aq^3 + \cdots + 
   \frac1{84}(11a^2 - 17897672a - 75168751256976)q^7 + \cdots, \\
  &q + 2^{15}q^2 + aq^3 + \cdots + 
   \frac1{216}(a^2 - 14315428a + 407319502919904)q^7 + \cdots,
\end{split}
\end{equation*}
respectively. Here each $a$ is a root of the characteristic polynomial
of $T_3$ for the corresponding weight, and is different at each
occurrence. From these, we deduce that the matrices for the Hecke
operator $T_7$ with respect to the bases in \eqref{equation: 10 basis}
are
$$ \extrarowheight3pt
\begin{array}{c|l} \hline\hline
k & B_k \\ \hline
 4 & -4 \\
 8 & 104 \\
12 & \begin{pmatrix} -22336 & 17346 \\
     305856 & 36428 \end{pmatrix} \\
16 & \begin{pmatrix} 628484 & 207270 \\
    14618880 & -1613416 \end{pmatrix} \\
20 & \begin{pmatrix} 49507856 & 1021650 \\
   136909440 & 33553436 \end{pmatrix} \\
24 & \begin{pmatrix} -826476664 & -2549118572 & 216037178 \\
   -4554273600 & -3184965196 & 546964950 \\
   27509870592 & 52096359168 & 934073672 \end{pmatrix} \\
28 & \begin{pmatrix} -91564144564 & 113245670860 & 5617780910 \\
  438283664640 & 412736094176 & -12502504350 \\
15396166471680 & -2162558865600 & -111965170324 \end{pmatrix} \\
32 & \begin{pmatrix} 4631981436536 & -203996300396 & 50284263050 \\
   -11858411062656 & 12584751782372 & 97180354950 \\
    18304356630528 & 78355740427776 &  4460404162424 \end{pmatrix} \\
\hline\hline
\end{array}
$$
Noticing that $F^m=f_{4m,d_{4m}}$ for any positive integer $m$, we
read from the matrices above that
$$
  \sum_{j=0}^7\frac{F_j}{F}=-\frac47, \qquad
  \sum_{j=0}^7\frac{F_j^2}{F^2}=\frac{104}{7^3}, \qquad
  \sum_{j=0}^7\frac{F_j^3}{F^3}=\frac1{7^5}(305856t^{-1}+36428),
$$
$$
  \sum_{j=0}^7\frac{F_j^4}{F^4}=\frac1{7^7}
  (14618880t^{-1}-1613416), \quad
  \sum_{j=0}^7\frac{F_j^5}{F^5}=\frac1{7^9}
  (136909440t^{-1}+33553436),
$$
\begin{equation*}
\begin{split}
  \sum_{j=0}^7\frac{F_j^6}{F^6}&=\frac1{7^{11}}
  (27509870592t^{-2}+52096359168t^{-1}+934073672), \\
  \sum_{j=0}^7\frac{F_j^7}{F^7}&=\frac1{7^{13}} 
  (15396166471680t^{-2}-2162558865600t^{-1}-111965170324), \\
  \sum_{j=0}^7\frac{F_j^8}{F^8}&=\frac1{7^{15}}
  (18304356630528t^{-2}+78355740427776t^{-1}+4460404162424).
\end{split}
\end{equation*}
Then the rational function $\Psi(z,x)$ in \eqref{equation: Psi} is
equal to
\begin{equation*}
\begin{split}
  \Psi(z,x)&=\frac1{7^{14}x^2}\Big(678223072849x^2z^8
  +387556041628x^2z^7+7909306972x^2z^6 \\
  &\qquad-(527663765132x^2+4114130940864x)z^5 \\
  &\qquad+(46199115214x^2-5360751039168x)z^4 \\
  &\qquad+(72916497220x^2-2228082208128x)z^3 \\
  &\qquad+(90698975500x^2-75419213184x+10905601867776)z^2 \\
  &\qquad+(-72866748500x^2+2516798571840x+9093300682752)z \\
  &\qquad+(13624725625x^2-487484222400x+10905601867776)\Big).
%  +2^2\cdot7^{13}x^2z^7
%   2^2\cdot7^{11}x^2z^6-(2^2\cdot7^{10}\cdot467x^2+2^6\cdot3^3\cdot7^9
\end{split}
\end{equation*}
We then set $\Psi'(z,y)=z^8\Psi(1/z,y)$ and compute the resultant
$R(x,y)$ of $\Psi(z,x)$ and $\Psi'(z,y)$ with respect to the variable
$z$. The numerator of $R(x,y)$ has two irreducible factors. To
determine which one corresponds to the modular equation $\Phi_7(x,y)$
of level $7$, we use the fact that the roots of $\Phi_7(x,x)$ should
be the coordinates of CM-points of discriminants $-3$, $-20$, $-27$,
$-35$, $-40$, $-52$, $-115$, $-180$, and $-280$, and these coordinates
are all rational numbers. In fact, the coordinates of these CM-points
were given in Table 3 of \cite{Elkies}. Those obtained numerically in
\cite{Elkies} were later verified by Errthum \cite{Errthum} using
Borcherds forms. (In general, if $\Phi_p(x,y)$ is the modular
equation of level $p$ for $X$, then the zeros of the polynomial
$\Phi_p(x,x)$ should be coordinates of CM-points of discriminants of
the form $(s^2-4p)/f^2$, $(4s^2-8p)/f^2$, $(25s^2-20p)/f^2$, and
$(100s^2-40p)/f^2$, subject to the condition that optimal embeddings
of imaginary quadratic order of given discriminant into the maximal
order in the quaternion algebra over $\Q$ of discriminant $10$ exist.)
That is, $\Phi_7(x,x)$ should factor into a product of linear factors
over $\Q$. Indeed, exactly one of the two irreducible factors has this
property. This determines $\Phi_7(x,x)$. The equation of $\Phi_7(x,x)$
is given in Section \ref{section: 10}.
\end{section}

\begin{section}{Modular equations for $X_0^6(1)/W_6$}
\label{section: 6}
In this section, we consider the Shimura curve $X_0^6(1)/W_6$. The
Hauptmodul $t$ is chosen such that it takes values $0$, $1$, and
$\infty$ at the CM-points of discriminant $-24$, $-4$, and $-3$,
respectively. For a prime $p\neq 2,3$, we let $\Phi_p(x,y)$ denote
the modular equation of level $p$ for the Shimura curve
$X_0^6(1)/W_6$. We have computed $\Phi_p(x,y)$ for primes up to $19$,
but because the coefficients are very big, here we only
list the equation for $p=7$. (The equation for $p=5$, with a slight
change of variables, is contained in Appendix A of
\cite{Yang-Schwarzian}.) Files containing equations of other levels
are available upon request.

We note that for $p=7$, an explicit cover $X_0^6(7)/W_6\to
X_0^6(1)/W_6$ have already been determined by Elkies \cite{Elkies}. It
is easier to use this explicit cover to obtain the modular equation.

Write $\Phi_7(x,y)$ as $a_8(x)y^8+\cdots+a_0(x)$. Then
\begin{equation*} \small
  a_0(x)=(262254607552729x^2-121636570723920x+501956755356672)^2,
\end{equation*}
\begin{equation*} \small
\begin{split}
a_1(x)&=-20948043194072943880152567521879704213891529442411744000000x^8\\
& - 972224700022233770983527054177244292949817108052713267466888x^7 \\
& +11273421679418251606098370957430742480477581392613807611381824x^6\\
& +1668909333131109923039277853641492658428121008149001087888384x^5 \\
& -63984186761131298882488446421883340458529413111323886398291968x^4\\
& +44207311496458731717072123962700099989994411322119128098799616x^3\\
& +16376343297097580660696811581414332932110848991772493800800256x^2\\
& -10266396414223579734319217958394551183170348131151337958146048x \\
& - 61535122440667017884769009139405588503624662039298456944640,
\end{split}
\end{equation*}
\begin{equation*} \small
\begin{split}
a_2(x)&=38652996866188686573666243770903500386081949440000000000000x^8\\
& - 754695394662453117523773206823864267539586547407817536000000x^7 \\
& +106920403569574588553159546946652660456141317611471702546206108x^6\\
&-304712233218468979235116193697598666589103401841722665615811904x^5\\
&+760829927783454753461831780276599490081935198148332681198238208x^4\\
&-838656978040090256476182485622824884772400584504768719373500416x^3\\
&+279281854044007197025412087758954247401641295020911801816973312x^2\\
&+16376343297097580660696811581414332932110848991772493800800256x\\
& + 155040070253203796103748431879063330519757479108093010771968,
\end{split}
\end{equation*}
\begin{equation*} \small
\begin{split}
a_3(x)&=-38038344753469448835190136766681303367680000000000000000000x^8\\
& + 5356342063300144660882808541232703156528549197824000000000000x^7\\
& +76321377218916949783840090111683558976051518322915065600000000x^6\\
&-474977723264104825683894637831760410975632622470023372974366136x^5\\
&+167509988385063268988566748828192523079606899212802760635355200x^4\\
&+929561753522872554487774286564540945788659043162410938085107712x^3\\
&-838656978040090256476182485622824884772400584504768719373500416x^2\\
&+ 44207311496458731717072123962700099989994411322119128098799616x \\
& - 100063182095188848153455102416032183737981837866201188925440,
\end{split}
\end{equation*}
\begin{equation*} \small
\begin{split}
a_4(x)&=21056325026580996459284526662615040000000000000000000000000x^8\\
& - 2676503627950279289618572142817040185729024000000000000000000x^7\\
& +74717592843661815288079677324740635068097597899264000000000000x^6\\
& +99843100080788043313984402843762031194115908069343957280000000x^5\\
&-643716270399690372725464450317965610492381617103712363036843130x^4\\
&+167509988385063268988566748828192523079606899212802760635355200x^3\\
&+760829927783454753461831780276599490081935198148332681198238208x^2\\
& - 63984186761131298882488446421883340458529413111323886398291968x \\
& + 127566357194950739797651934018366979046647830016645952241664,
\end{split}
\end{equation*}
\begin{equation*} \small
\begin{split}
a_5(x)&=-6216447485624191385037963264000000000000000000000000000000x^8\\
& - 969897592146461089985813356429049856000000000000000000000000x^7 \\
& - 4716808719470806728440581205159667124740096000000000000000000x^6\\
&+100946931718713025406590060876901549085887873170944000000000000x^5\\
& +99843100080788043313984402843762031194115908069343957280000000x^4\\
&-474977723264104825683894637831760410975632622470023372974366136x^3\\
&-304712233218468979235116193697598666589103401841722665615811904x^2\\
& + 1668909333131109923039277853641492658428121008149001087888384x \\
& - 52279464855900563088298143909713657632323089372536352686080,
\end{split}
\end{equation*}
\begin{equation*} \small
\begin{split}
a_6(x)&=764699349893278334976000000000000000000000000000000000000x^8\\
& - 15347209366323527757160513536000000000000000000000000000000x^7 \\
& - 1253143615853758839349272867028205568000000000000000000000000x^6\\
& - 4716808719470806728440581205159667124740096000000000000000000x^5\\
& +74717592843661815288079677324740635068097597899264000000000000x^4\\
& +76321377218916949783840090111683558976051518322915065600000000x^3\\
&+106920403569574588553159546946652660456141317611471702546206108x^2\\
& + 11273421679418251606098370957430742480477581392613807611381824x \\
& + 42321163987340612844969199376648358822165115192172887352832,
\end{split}
\end{equation*}
\begin{equation*} \small
\begin{split}
a_7(x)&=997738013984972341248000000000000000000000000000000000000x^7\\
& - 15347209366323527757160513536000000000000000000000000000000x^6 \\
& - 969897592146461089985813356429049856000000000000000000000000x^5\\
& - 2676503627950279289618572142817040185729024000000000000000000x^4\\
& + 5356342063300144660882808541232703156528549197824000000000000x^3\\
& - 754695394662453117523773206823864267539586547407817536000000x^2 \\
& - 972224700022233770983527054177244292949817108052713267466888x \\
& - 8775937874145070526339174290824946318421255782225802867520,
\end{split}
\end{equation*}
\begin{equation*} \small
a_8(x)=(3024000000x-4097152081)^6.
\end{equation*}

We next give a short list of irrational singular moduli obtained by
factorizing $\Phi_p(x,x)$ for $p\le 19$. Note that the norms of these
coordinates were already computed in \cite{Errthum}. Our computation
yields their exact values, not just their norms. (Note that all the
rational singular moduli for $X_0^6(1)/W_6$ were numerically
determined in \cite{Elkies} and later verified in \cite{Errthum}.)
% The following table gives the
%coordinates of CM-points obtained by factorizing $\Phi_p(x,x)$ for
%$p\le 19$. In particular, it verifies the entries in Table 1 of
%\cite{Elkies}, which were obtained by numerical computation, except
%for the discriminants $-267$, $-232$, $-708$, and $-163$, which
%require modular equations of level $23$, $29$, $31$, and $41$,
%respectively, to verify. The entries are listed roughly in the
%ascending order of heights of the norms of the coordinates over $\Q$.
%Here in the third column we give the smallest $p$ such that the
%coordinate of the CM-point of discriminant $d$ appears as a zero of
%$\Phi_p(x,x)$.

$$ \extrarowheight3pt \tiny
\begin{array}{c|l|l} \hline\hline
d & \text{coordinates} & \text{norms} \\ \hline
%-3 & \infty & 1\\
%-4 & x-1 & 1 \\
%-24 & x & 1 \\
%-84 & 27x + 169 & 5 \\
%-40 & 125x - 2312 & 5 \\
%-51 & 1024x + 1377 & 5 \\
%-19 & 1024x - 3211 & 5 \\
%-120& 3375x - 5776 & 5 \\
%-52 & 15625x - 6877 & 7 \\
%-132& 15625x - 13689 & 7 \\
%-75 & 138240x - 152881 & 7 \\
%-168& 15625x + 701784 & 7 \\
%-43 & 16000000x - 21250987 & 11 \\
%-228& 11390625x - 66863329 & 11 \\
%-88 & 20796875x - 15545888 & 11 \\
%-123& 16000000x + 296900721 & 11 \\
-264& 26198073x^2 - 53148528x + 42719296 & 2^6 19^2 43^2/3^9 11^3 \\
-276& 1771561x^2 + 328736070x - 206586207 & 3^8 23^1 37^2/11^6 \\
%-100& 1771561x - 421850521 & 13 \\
%-147& 3024000000x + 1073152081 & 13 \\
%-312& 27680640625x - 27008742384 & 13 \\
%-67 & 1024000000x - 77903700667 & 17 \\
-136& 8703679193x^2 + 9087757328x + 52236012608 & 2^6 13^4 41^2/11^6 17^2 \\
%-148& 377149515625x - 69630712957 & 19 \\
%-372& 747377296875x + 455413074649 & 17 \\
%-408& 55962140625x + 32408609436736 & 17 \\
-195& 1048576000000x^2 - 1812440448000x + 884218628241
    & 3^8 13^2 19^2 47^2/2^{26} 5^6 \\
-420& 377149515625x^2 - 1357411709250x + 3732357388761
    & 3^8 23^2 61^2/5^6 17^4 \\
-456& 42761175875209x^2 - 43342656704496x + 31050765799488
    & 2^6 3^9 19^1 67^2/11^6 17^4 \\
-219& 1320903770112x^2 + 125096383466496x - 128030514913249
    & 13^4 23^2 41^2 71^2/2^{26}3^9 \\
\hline\hline
\end{array}
$$
\end{section}

\begin{section}{Modular equations for $X_0^{10}(1)/W_{10}$}
\label{section: 10}
In this section, we consider the Shimura curve
$X_0^{10}(1)/W_{10}$. The Hauptmodul $t$ is chosen such that it takes
values $0$, $2$, $27$, and $\infty$ at the CM-points of discriminant
$-3$, $-20$, $-40$, and $-8$, respectively. For a prime $p\neq 2,5$,
we let $\Phi_p(x,y)$ denote the modular equation of level $p$ for the
Shimura curve $X_0^{10}(1)/W_{10}$. We have computed $\Phi_p(x,y)$ for
primes up to $23$, but here we only list the equations for
$p=7$. Files containing equations of other levels are available upon
request. (The equation for $p=3$ is given in Section 5 of
\cite{Yang-Schwarzian}.)

Write $\Phi_7(x,y)$ as $a_8(x)y^8+\cdots+a_0(x)$. Then
\begin{equation*} \small
\begin{split}
  a_0(x)=x^2(278055625x^2-9948657600x+222563303424)^3,
\end{split}
\end{equation*}
\begin{equation*} \small
\begin{split}
  a_1(x)&=-22595841233117014160156250x^8 \\
  & + 28984697923550260459599062500x^7 \\
  & - 1645302326393970545433073200000x^6 \\
  & + 6642716791333154856889291776000x^5 \\
  & + 602202478597261111037649172561920x^4 \\
  & - 4910763943297233363215784691630080x^3 \\
  & - 8496599535482361705366127525232640x^2 \\
  & + 9192477793078594137923003873230848x,
\end{split}
\end{equation*}
\begin{equation*} \small
\begin{split}
  a_2(x)&=227107188088738387105484375x^8 \\
  & - 10394755051722304542380018750x^7 \\
  & + 2778503946540202482367606147300x^6 \\
  & - 87342791194115920501573603009536x^5 \\
  & + 509955665442187448278395883376640x^4 \\
  & + 4112275508763979970115841213071360x^3 \\
  & - 2548787139523595922308471728373760x^2 \\
  & - 8496599535482361705366127525232640x \\
  & + 11024545045545306955620945204609024,
\end{split}
\end{equation*}
\begin{equation*} \small
\begin{split}
  a_3(x)&=-150467239701697875855937500x^8 \\
  & - 66608231302181029579976003090x^7 \\
  & + 1779457212256262244783949440310x^6 \\
  & + 2537754854567828839953906813980x^5 \\
  & - 567333242539975544914354048540992x^4 \\
  & - 239926105649803444132294072811520x^3 \\
  & + 4112275508763979970115841213071360x^2 \\
  & - 4910763943297233363215784691630080x \\
  & - 1478403072292996522939287129292800,
\end{split}
\end{equation*}
\begin{equation*} \small
\begin{split}
  a_4(x)&=591743143857926286378268151x^8 \\
  & + 3202866389133529516632265220x^7 \\
  & - 354063672667325366031194372695x^6 \\
  & + 22584679924916438377343486492770x^5 \\
  & + 91353706850814722447426682881030x^4 \\
  & - 567333242539975544914354048540992x^3 \\
  & + 509955665442187448278395883376640x^2 \\
  & + 602202478597261111037649172561920x \\
  & + 107405110735755767009079459840000,
\end{split}
\end{equation*}
\begin{equation*} \small
\begin{split}
  a_5(x)&=-178678323923978351074342650x^8 \\
  & + 71993226806139685483910592x^7 \\
  & - 382859191369766872635900582760x^6 \\
  & - 3892236690880311638995825770460x^5 \\
  & + 22584679924916438377343486492770x^4 \\
  & + 2537754854567828839953906813980x^3 \\
  & - 87342791194115920501573603009536x^2 \\
  & + 6642716791333154856889291776000x \\
  & - 4678710992156906294833152000000,
\end{split}
\end{equation*}
\begin{equation*} \small
\begin{split}
  a_6(x)&=38853800265463031036085625x^8 \\
  & + 2535839489102333226932059850x^7 \\
  & + 55519696363179781200215763218x^6 \\
  & - 382859191369766872635900582760x^5 \\
  & - 354063672667325366031194372695x^4 \\
  & + 1779457212256262244783949440310x^3 \\
  & + 2778503946540202482367606147300x^2 \\
  & - 1645302326393970545433073200000x \\
  & + 134184722882776668446400000000,
\end{split}
\end{equation*}
\begin{equation*} \small
\begin{split}
  a_7(x)&=-3904463975228746224750000x^8 \\
  & - 279343739152597734767051250x^7 \\
  & + 2535839489102333226932059850x^6 \\
  & + 71993226806139685483910592x^5 \\
  & + 3202866389133529516632265220x^4 \\
  & - 66608231302181029579976003090x^3 \\
  & - 10394755051722304542380018750x^2 \\
  & + 28984697923550260459599062500x \\
  & - 2307539315546608933125000000,
\end{split}
\end{equation*}
\begin{equation*} \small
\begin{split}
  a_8(x)&=(528806915000x^4-3691767131325x^3+23850535042201x^2\\
  &\qquad\qquad-2436693984375x+4636577546875)^2.
\end{split}
\end{equation*}

Here we give a short table of irrational singular moduli obtained by
factorizing $\Phi_p(x,x)$ for $p\le 23$. Note that the norms of these
singular moduli were already computed in \cite{Errthum} using
Borcherds forms. (All the rational singular moduli for
$X_0^{10}(1)/W_{10}$ were numerically determined in \cite{Elkies} and
later verified in \cite{Errthum}.)
% The next table gives the
%coordinates of CM-points obtained by factorizing $\Phi_p(x,x)$ for
%$p\le 19$. In particular, it verifies the entries in Table 3 of
%\cite{Elkies}, which were obtained by numerical computation, except
%for the discriminants $-232$ and $-163$, which
%require modular equations of level $29$ and $41$,
%respectively, to verify. Note that the CM-point of discriminant $-235$
%was missing in Table 3 of \cite{Elkies}.

%Here in the third column we give the smallest
%$p$ such that the coordinate of the CM-point of discriminant $d$
%appears as a zero of $\Phi_p(x,x)$.

$$ \extrarowheight3pt \tiny
\begin{array}{c|l|l} \hline\hline
d & \text{coordinates} & \text{norms} \\ \hline
%-8 & \infty & 1 \\
%-3 & x & 1\\
%-20 & x-2 & 1 \\
%-40 & x-27 & 1 \\
%-120& 49x + 27 & 3 \\
%-52 & 25x + 54 & 7 \\
%-35 & 7x - 64 & 3 \\
%-72 & 147x - 125 & 11 \\
%-27 & 25x + 192 & 7 \\
-68 & 25x^2 - 164x + 500 & 2^2 5^1 \\
%-43 & 1225x - 1728 & 11 \\
%-180& 169x + 2662 & 7 \\
%-88 & 98x - 3375 & 11 \\
%-115& 3887x - 13824 & 7 \\
-260& 31213x^2 - 7252x + 19652 & 2^2 17^3/7^4 13^1 \\
%-280& 7406x - 35937 & 7 \\
%-67 & 8281x + 216000 & 17 \\
%-340& 41209x - 657018 & 11 \\
%-148& 207025x - 71874 & 19 \\
-360& 3844x^2 + 16299x + 132651 & 3^3 17^3/2^2 31^2 \\
-152& 1250x^3 - 62475x^2 + 64048x - 166375 & 11^3/2^1 5^1 \\
-132& 4225x^2 - 27108x + 364500 & 2^2 3^6 5^1/13^2 \\
-168& 122500x^2 - 647325x + 970299 & 3^6 11^3/2^2 5^4 7^2 \\
-195& 142129x^2 + 12489984x - 2985984 & 2^{12} 3^6/13^2 29^2 \\
-155& 74431x^2 - 1204224x - 5451776 & 2^{12} 11^3/7^4 31^1 \\
%-235& 3152807x - 8489664 & 13 \\
%-520& 11257064x - 658503 & 13 \\
-440& 114244x^3 - 13732264x^2 + 20581253x - 16194277 & 11^3 23^3/2^2 13^4 \\
-83 & 105625x^3 - 2894400x^2 + 30224384x + 32768000
    & 2^{18}/5^1 13^2 \\
-420& 67081x^2 + 431676x + 71118324 & 2^2 3^6 29^3/7^2 37^2\\
-580& 13766909x^2 - 22395636x + 200973636 & 2^2 3^6 41^3/13^2 29^1 53^2 \\
-315& 72812089x^2 + 88682944x + 398688256 & 2^{15} 23^1/7^2 53^2 \\
-820& 3158777209x^2 - 19708111836x - 598885164
    & 2^2 3^6 59^3/7^4 13^4 37^2 \\
-355& 253887551x^2 - 6128728704x + 3974344704
    & 2^{12} 3^6 11^3/31^2 61^2 71 \\
-660& 4726150009x^2 - 10741876236x + 8174192436
    & 2^2 3^9 47^3/7^4 23^2 61^2 \\
%-760& 237375649x - 13772224773 & 19 \\
-680& 995781136x^3 - 8234027816x^2 + 166339878257x - 450688409963
    & 11^3 17^3 41^3/2^4 7^6 23^2\\
-435& 4460503369x^2 - 298408090176x + 25350000869376
    & 2^{18} 3^9 17^3/7^4 29^2 47^2 \\
-920& \substack{287478606136x^5 - 3389457456860x^4 - 4652502434310x^3\\
    \quad+ 216562501495x^2 + 50275785584380x - 60944362239977}
    & 23^2 59^3/2^3 29^1 47^2 \\
\hline\hline
\end{array}
$$
\end{section}

\end{document}